\documentclass{siamart190516}

\usepackage{lipsum}
\usepackage{amsfonts}
\usepackage{graphicx}
\usepackage{epstopdf}
\usepackage{algorithmic}
\ifpdf
  \DeclareGraphicsExtensions{.eps,.pdf,.png,.jpg}
\else
  \DeclareGraphicsExtensions{.eps}
\fi

\newsiamremark{hypothesis}{Hypothesis}
\crefname{hypothesis}{Hypothesis}{Hypotheses}
\newsiamthm{claim}{Claim}

\headers{Adaptive Deep Neural Network}{Z. Cai, J. Chen and M. Liu}

\title{Self-Adaptive Deep Neural Network: \\ Numerical Approximation to Functions and PDEs \thanks{This work was supported in part by the National Science Foundation
under grant DMS-2110571.}}

\author{Zhiqiang Cai\thanks{Department of Mathematics, Purdue University, 150 N. University Street, West Lafayette, IN 47907-2067 
  (\email{caiz@purdue.edu}, \email{chen2042@purdue.edu}).}
\and Jingshuang Chen\footnotemark[2]
\and Min Liu\thanks{School of Mechanical Engineering, Purdue University, 585 Purdue Mall,
West Lafayette, IN 47907-2088(\email{liu66@purdue.edu}). }}

\usepackage{amsopn}

\ifpdf
\hypersetup{
  pdftitle={Adaptive DNN},
  pdfauthor={Z.Cai, J.Chen, M.Liu}
}
\fi

\DeclareUnicodeCharacter{2212}{-}
\usepackage[utf8]{inputenc}
\usepackage{bm}
\usepackage{mathtools}
\usepackage{indentfirst}
\usepackage{subfigure}
\usepackage{wrapfig}
\usepackage{tcolorbox}
\usepackage{color}
\usepackage[export]{adjustbox}
\usepackage{makecell}
\usepackage{textcomp}
\usepackage{algorithm }
\usepackage{algorithmic}
\usepackage{booktabs}

\newcommand{\R}{\mathbb{R}}

\usepackage{graphicx}
\usepackage{diagbox}
\usepackage{pifont}
\newcommand{\vertiii}[1]{{\left\vert\kern-0.25ex\left\vert\kern-0.25ex\left\vert #1 
    \right\vert\kern-0.25ex\right\vert\kern-0.25ex\right\vert}}

\newcommand{\btheta}{\mbox{\boldmath${\theta}$}}

\newcommand{\bomega}{\mbox{\boldmath${\omega}$}}
\newcommand{\bxi}{\mbox{\boldmath$\xi$}}

\newcommand{\bff}{{\bf f}}

\newcommand{\bb}{{\bf b}}
\newcommand{\bv}{{\bf v}}

\newcommand{\bx}{{\bf x}}

\newcommand{\cI}{{\cal I}}
\newcommand{\cK}{{\cal K}}
\newcommand{\cM}{{\cal M}}
\newcommand{\cQ}{{\cal Q}}
\newcommand{\cS}{{\cal S}}
\newcommand{\cT}{{\cal T}}

\def\bb{{\bf b}}
\def\bc{{\bf c}}
\def\bl{{\bf l}}
\def\bo{{\bf o}}
\def\bx{{\bf x}}
\def\by{{\bf y}}

\def\bzero{{\bf 0}}

\def\cC{{\cal C}}
\def\cL{{\cal L}}
\def\cM{{\cal M}}
\def\cN{{\cal N}}
\def\cP{{\cal P}}
\def\cS{{\cal S}}
\def\cT{{\cal T}}

\def\cP{{\cal P}}

\begin{document}

\maketitle
\begin{abstract}
Designing an optimal deep neural network for a given task is important and challenging in many machine learning applications. To address this issue, we introduce a self-adaptive algorithm: the adaptive network enhancement (ANE) method, written as loops of the form
\begin{center}
$\mbox{\bf train }  \rightarrow \,\,\,
\mbox{\bf estimate }  \rightarrow \,\,\,
\mbox{\bf enhance}.$
\end{center}

\noindent Starting with a small two-layer neural network (NN), the step {\bf train} is to solve the optimization problem at the current NN; the step {\bf estimate} is to compute {\it a posteriori} estimator/indicators using the solution at the current NN; the step {\bf enhance} is to 
add new neurons to the current NN.

Novel network enhancement strategies based on the computed estimator/indicators are developed in this paper to determine how many new neurons and when a new layer should be added to the current NN. The ANE method provides a natural process for obtaining a good initialization in training the current NN; in addition, we introduce an advanced procedure on how to initialize newly added neurons for a better approximation. We demonstrate that the ANE method
can automatically design a nearly minimal NN for learning functions exhibiting sharp transitional layers as well as discontinuous solutions of hyperbolic partial differential equations.
\end{abstract}

\begin{keywords}
Self-adaptivity, Advection-reaction equation, Least-squares approximation, Deep neural network, ReLU activation
\end{keywords}

\section{Introduction}
Deep neural network (DNN) has achieved astonishing performance in computer vision, natural language processing, and many other artificial intelligence (AI) tasks (see, e.g., \cite{Collobert2011,NIPS2012_4824,Goodfellow2016}). This success encourages wide applications to other fields, including recent studies of using DNN models to learn solutions of partial differential equations (PDEs) (see, e.g., \cite{Berg18, CAI2020, Cai2021nonlinear, Weinan18, Karniadakis19, Sirignano18}). The phenomenal performance on many AI tasks comes at the cost of high computational complexity. Accordingly, designing efficient network architectures for DNN is an important step towards enabling the wide deployment of DNN in various applications.

Studies and applications of neural network (NN) may be traced back to the work of Hebb \cite{Hebb1949} in the late 1940's and Rosenblatt \cite{Rosenblatt1958} in the 1950's. 
DNN produces a new class of functions through compositions of linear transformations and activation functions. This class of functions is extremely rich. For example, it contains piece-wise polynomials, which are the footing of spectral elements, and continuous and discontinuous finite element methods for computer simulations of complex physical, biological, and human-engineered systems. It approximates polynomials of any degree with exponential efficiency, even using simple activation functions like ReLU. More importantly, a neural network function can automatically adapt to a target function or the solution of a PDE.

Despite great successes of DNN in many practical applications, it is widely accepted that approximation properties of DNN are not yet well understood and that understanding on why and how they work could lead to significant improvements in many machine learning applications. First, some empirical observations suggest that deep network can approximate many functions more accurately than shallow network, but rigorous study on the theoretical advantage of deep network is scarce. Therefore, even in the manual design of network models, the addition of neurons along depth or width is {\it ad-hoc}. Second, current methods on design of the architecture of DNN in terms of their width and depth are empirical. Tuning of depth and width is tedious, mainly from experimental results in ablation studies which typically require domain knowledge about the underlying problem. Third, there is a tendency in practice to use over-parametrized neural networks; this leads to a high-dimensional nonlinear optimization problem which is much more difficult to train than a low-dimensional one. These considerations suggest that a fundamental, open question to be addressed in scientific machine learning is: what is the optimal network model required, in terms of width, depth, and the number of parameters, to learn data, a function, or the solution of a PDE within some prescribed accuracy?


To address this issue, we introduce a self-adaptive algorithm: the adaptive network enhancement (ANE) method, written as loops of the form
\begin{center}
$\mbox{\bf train }  \rightarrow \,\,\,
\mbox{\bf estimate }  \rightarrow \,\,\,
\mbox{\bf enhance}.$
\end{center}

\noindent Starting with a small two-layer NN, the step {\bf train} is to solve the optimization problem of the current NN; the step {\bf estimate} is to compute {\it a posteriori} estimator/indicators using the solution at the current NN; the step {\bf enhance} is to 
add new neurons to the current NN. 
This adaptive algorithm learns not only from given information (data, function, PDE) but also from the current computer simulation, and it is therefore a learning algorithm at a level which is more advanced than common machine learning algorithms.

To develop an efficient ANE method, we need to address the following essential questions at each adaptive step when the current NN is not sufficient for the given task: 
\begin{itemize}
    \item [(a)] how many new neurons should be added?
    \item [(b)] when should a new layer be added?
\end{itemize}
For a two-layer NN, we proposed the ANE method (see Algorithm 4.1) for learning a given function in \cite{LiuCai1} and the solution of a given self-adjoint elliptic PDEs through the Ritz formulation in \cite{LiuCai2}. In the case of a two-layer NN, question (b) is irrelevant and question (a) was addressed by introducing a network enhancement strategy that decides the number of new neurons to be added in the first hidden layer. This strategy is based on the physical partition of the current computer simulation determined by the {\it a posteriori} error indicators (see Algorithm 3.1).

For a multi-layer NN, it is challenging to address both the questions. First, the role of a neuron in approximation at a hidden layer varies and depends on which hidden layer the neuron is located. Second, there is almost no understanding on the role of a specific hidden layer in approximation and, hence, we have no {\it a priori} approximation information for determining when a new layer should be added. 
To resolve question (a) for a multi-layer, we will exploit the geometric property of the current computer simulation and introduce a novel enhancement strategy (see Algorithm 4.2) that determines the number of new neurons to be added at a hidden layer other than the first hidden layer. For question (b), we will introduce a computable quantity to measure the improvement rate of two consecutive NNs per the relative increase of parameters. When the improvement rate is small, then a new layer is started. 

Training DNN, i.e., determining the values of the parameters of DNN, is a problem in nonlinear optimization. This high dimensional, nonlinear optimization problem tends to be computationally intensive and complicated and is usually solved iteratively by the method of gradient descent and its variations (see \cite{BoCuNo2018}). In 
general, a nonlinear optimization has many solutions, and the desired one is obtained only if we start from a close enough first approximation. A common way to obtain a good initialization is by the method of continuation \cite{AllgowerGeorg1990}. 
The ANE method provides a natural process for obtaining a good initialization. Basically, the approximation at the previous NN is already a good approximation to the current NN in the loops of the ANE method. To provide a better approximation, 
we initialize the weights and bias of newly added neurons at the first hidden layer by using the physical partition of the domain (see Section 3 and \cite{LiuCai1} for details); in this paper we introduce an advanced procedure on how to initialize newly added neurons at a hidden layer that is not the first hidden layer.

For simplicity of presentation, the ANE method for a multi-layer NN is first described for learning a given function through the least-squares loss function (see Section 4). The method is then applied to learn solutions of linear advection-reaction equations through
the least-squares neural network (LSNN) method introduced in \cite{Cai2021linear} (see Section 7). 
We demonstrate that the ANE method can automatically design a nearly minimal NN for learning functions exhibiting sharp transitional layers as well as discontinuous solutions of hyperbolic PDEs.

Recently, there is growing interest in automatic machine learning (AutoML) in an effort of replacing a manual design process of architectures by human experts. Neural architecture search (NAS) method (see a survey paper \cite{NAS_survey} and reference therein) presents a general methodology at high level on AutoML. It consists of three components: search space, search strategy, and performance estimation strategy. Usually the resulting algorithm is computationally intensive because it explores a wide range of potential network architectures. Nevertheless, the NAS outperforms manually designed architectures in accuracy on some tasks such as image classification, object detection, or semantic segmentation. The NAS based on the physics-informed neural network is recently used for solving stochastic groundwater flow problem in \cite{Guo2020}.

The paper is organized as follows. DNN, the best least-squares (LS) approximation to a given function using DNN, and the discrete counterpart of the best LS approximation are introduced in Section~2. The physical partition for a DNN function is described in Section 3. The ANE method, initialization of parameters at different stage, and numerical experiments are presented in Sections 4, 5, and 6, respectively. Finally, application of the ANE method to the linear advection-reaction equation is given in Section 7.

 \section{Deep Neural Network and Least-squares Approximation}

A deep neural network defines a function of the form
\begin{equation}\label{NN}
 \by=\cN(\bx)=\bomega^{L}\cdot\left(N^{(L-1)} \circ \cdots \circ N^{(1)}(\bx)\right) - b^{L}
 :\, \bx\in\R^{d}
\longrightarrow \by=\cN(\bx)\in\R,
\end{equation} 
where $d$ is the dimension of input $\bx$, $\bomega^{(L)}\in \R^{n_{L-1}}$, $b^{(L)}\in \R$, the symbol $\circ$ denotes the composition of functions, and $L$ is the depth of the network. 
For $l=1,\cdots,L-1$, the $N^{(l)}\!: \R^{n_{l-1}}\! \rightarrow \R^{n_{l}}$ is called
the $l^{th}$ hidden layer of the network defined by
\begin{equation}\label{layerdef}
  N^{(l)}(\bx^{(l-1)})
  = \sigma (\bomega^{(l)}\bx^{(l-1)}-\bb^{(l)})
  \quad\mbox{for } \bx^{(l-1)}\in \R^{n_{l-1}},
\end{equation}
where $\bomega^{(l)} 
\in \R^{n_{l}\times n_{l-1}}$, $\bb^{(l)}\in \R^{n_{l}}$, $\bx^{(0)}=\bx$, and $\sigma(t)=\max\{0, t\}^p$ with positive integer $p$ is the activation function and its application 
to a vector is defined component-wise. This activation function is referred to as a spline activation ReLU$^p$. When $p=1$, $\sigma(t)$ is the popular rectified linear unit (ReLU). There are many other activation functions  such as (logistic, Gaussian, arctan) sigmoids (see, e.g., \cite{pinkus1999}).

Let ${\small\btheta}$ denote all parameters to be trained, i.e., the weights $\{\bomega^{(l)}\}_{l=1}^L$ and the bias $\{\bb^{(l)}\}_{l=1}^L$. Then the total number of parameters is given by
\begin{equation}\label{DoF}
   N=M_d(L) =\sum^L_{l=1} n_{l}\times (n_{l-1}+1).
\end{equation}
Denote the set of all DNN functions by
\[
\cM_N({\small\btheta},L)=\big\{
\bomega^{L}\cdot\left(N^{(L-1)} \circ \cdots \circ  N^{(1)}(\bx)\right) - b^{L} :\,  \bomega^{(l)} 
\in \R^{n_{l}\times n_{l-1}},\,\, \bb^{(l)}\in \R^{n_{l}} \mbox{ for } l=1,...,L-1
\big\}.
\]

Let $f(\bx)\in \R$ be a given target function defined in a domain $\Omega\in\R^d$. Training DNN to learn the function $f(\bx)$ using least-squares loss function amounts to solve the following best least-squares approximation: find $f_N(\bx ; {\small\btheta}^*)\in \cM_N({\small\btheta},L)$ such that
 \begin{equation}\label{L2App}
  \|f(\cdot)-f_N(\cdot;{\small\btheta}^*)\|=\min_{v\in \cM_N({\small\btheta},L)} \|f-v\|
  =\min_{{\scriptsize\btheta}\in\R^{N}} \|f(\cdot)-v(\cdot;{\small\btheta})\|,
  \end{equation}
where $v(\bx;{\small\btheta})\in \cM_N({\small\btheta},L)$ is of the form
 \[
 v(\bx;{\small\btheta})= \bomega^{L}\cdot\left(N^{(L-1)} \circ \cdots \circ  N^{(1)}(\bx)\right) - b^{L}
 \]
with $N^{(l)}(\cdot)$ defined in (\ref{layerdef}), and $\|v(\cdot)\|=\left(\int_\Omega v^2(\bx)\,d\bx\right)^{1/2}$ is the $L^2(\Omega)$ norm.

Let $\cI$ be the integral operator over the domain $\Omega$ given by
 \begin{equation}\label{integral}
 \cI(f)=\int_\Omega f(\bx)\, d\bx.
 \end{equation}
Let $\cT=\{K\, :\, K\mbox{ is an open subdomain of } \Omega\}$ be a partition of the domain $\Omega$, i.e.,
union of all subdomains of ${\cal T}$ equals to the whole domain $\Omega$ and that any two distinct subdomains of ${\cal T}$ have no intersection.
Let $\cQ_{_\cT}$ be a quadrature operator based on the partition $\cT$, i.e., 
 $\cI (v) \approx \cQ_{_\cT} \big(v\big)$, such that 
 \[
 \|v\|_{_\cT}=\sqrt{(v,v)_{_\cT}}=\sqrt{\cQ_{_\cT} \big(v^2\big)}
 \]
defines a weighted $l_2$-norm.
The best discrete least-squares approximation with numerical integration over the partition $\cT$ is to find $f_{_\cT}(\bx ; {{\small\btheta}^*_{_\cT}})\in \cM_N({\small\btheta},L)$ such that
 \begin{equation}\label{L2App-d}
  \|f(\cdot)-f_{_\cT}(\cdot ; {{\small\btheta}^*_{_\cT}})\|_{_\cT}
  =\min_{v\in \cM_N({\small\btheta},L)} \|f-v\|_{_\cT} 
  =\min_{{\scriptsize\btheta}\in\R^{N}}
  \|f(\cdot)-v(\cdot;{\small\btheta})\|_{_\cT}.
  \end{equation}
  
\begin{theorem}\label{error-bound1}
Assume that there exists a positive constant $\alpha$ such that $\alpha\, \|v\|^2 \leq \|v\|_{_\cT}^2$ for all $v\in \cM_{2N}\equiv\cM_N({\small\btheta},L)\oplus \cM_N({\small\btheta},L)$.
Let $f_{_\cT}$ be a solution of {\em (\ref{L2App-d})}.
Then there exists a positive constant $C$ such that
\begin{equation}\label{error-bound}
 \quad  \qquad C\,\|f-f_{_\cT}\|
 \leq \!\! \inf_{v\in \cM_N({\small\btheta},L)}\!\! \left\{\|f-v\| + \!\!\sup _{w\in \cM_{2N}}\!\! \dfrac{|(\cI-\cQ_{_\cT})(vw)|}{\|w\|}\right\}+ \!\!\sup _{w\in \cM_{2N}}\!\! \dfrac{|(\cI-\cQ_{_\cT})(fw)|}{\|w\|}.
\end{equation}
\end{theorem}

\begin{proof}
The theorem may be proved in a similar fashion as that of Theorem 4.1 in \cite{LiuCai1}.
\end{proof}

\section{Physical Partition (PP)}

As seen in \cite{LiuCai1}, the physical partition of the current NN approximation plays a critical role in the ANE method for a two-layer NN. As we shall see, it is essential for our self-adaptive multi-layer NN as well.
For simplicity of presentation, we consider ReLU activation function only in this section. The idea of our procedure for determining the physical partition can be easily extended to other activation functions even though the corresponding geometry becomes complex. 

For any function $v\in \cM_N(\btheta,k)$ with $k\ge 2$, it is easy to see that $v$ is a continuous piece-wise linear function with respect to a partition $\mathcal{K}^{(k-1)}$ of the domain $\Omega$. This partition is referred to as the {\it physical partition} of the function $v$ in $\cM_N(\btheta,k)$. This section describes how to determine the physical partition $\mathcal{K}^{(k-1)}$ of a function in $\cM_N(\btheta,k)$. To this end, for $l=1, \, \cdots ,\,k-1$, denote by $\cK^{(l)}$ the physical partition of the first $l$ layers.

To determine the physical partition $\cK^{(1)}$, notice that a two-layer NN with $n_1$ neurons generates the following set of functions:
\begin{equation}\label{ReLU-n}
 {\cal M}_N({\small\btheta},2) = \left\{\sum_{i=1}^{n_1} \omega^{(2)}_i\sigma(\bomega^{(1)}_i\cdot \bx -b^{(1)}_i) -b^{(2)}\, :\,  
\omega^{(2)}_i,\, b^{(1)}_i,\, b^{(2)}\in \R,\,\, \bomega^{(1)}_i\in \cS^{d-1}\right\},
 \end{equation}
where $\cS^{d-1}$ is the unit sphere in $\R^d$. The ${\cal M}_N({\small\btheta},2)$ may be viewed as an extension of the free-knot spline functions \cite{Schumaker} to multi-dimension, and its free breaking hyper-planes are
 \begin{equation}\label{planes}
 {\cal P}_j:\, \,\bomega_j^{(1)}\cdot \bx-b_j^{(1)}=0
 \quad\mbox{for } j=1,\,...,\,n_1.
 \end{equation}
This suggests that the physical partition $\cK^{(1)}$ is formed by the boundary of the domain $\Omega$ and the hyper-planes $\left\{\bomega^{(1)}_j\cdot \bx-b^{(1)}_j=0\right\}_{j=1}^{n_1}$. 


Next, we describe how to form the physical partition $\cK^{(l)}$. Our procedure is based on the observation that for $l=2,\, ...\, , k-1$, the $\cK^{(l)}$ may be viewed as a refinement of the $\cK^{(l-1)}$.
For $j=1, \, \cdots, \, n_l$, denote the function generated by the $j^{\text{th}}$ neuron at the $l^{\text{th}}$-layer without the activation function by 
\[
g_j^{(l)}(\bx) = \sum _{i=1}^{n_{l-1}} \omega^{(l)}_{ij} \sigma  \left(\bomega^{(l-1)}_i\cdot \bx^{(l-2)}-b^{(l-1)}_i\right) - b^{(l)}_j,
\]
\begin{wrapfigure}{r}{2.2in}
\includegraphics[width=2.2in]{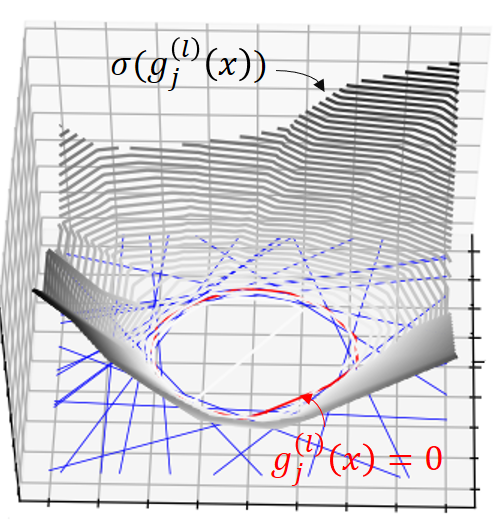} \vspace{-10pt}
\caption{\small Breaking lines generated by the $j^{\text{th}}$ neuron of the $l^{\text{th}}$-layer} 
\label{f_breaklines}
\vspace{-10pt}
\end{wrapfigure}
where $\bx^{(l-2)} = N^{(l-2)} \circ \cdots \circ  N^{(1)}(\bx)$. It is clear that the functions $g_j^{(l)}(\bx)$ are continuous piece-wise linear functions with respect to the physical partition $\cK^{(l-1)}$. The action of the activation function on $g_j^{(l)}(\bx)$, i.e., $\sigma (g_j^{(l)}(\bx)) = \max \{0, g_j^{(l)}(\bx)\}$ for $j=1, \, \cdots, \, n_l$, generate $n_l$ continuous piece-wise linear functions with respect to a refined partition of $\cK^{(l-1)}$. Therefore, the refinement is created by the activation function through replacing negative values of $g_j^{(l)}(\bx)$ by zero (see Fig. \ref{f_breaklines} for illustration). In other words, the refinement is done by all new hyper-planes satisfying
 \begin{equation}\label{new-hplanes}
    g_j^{(l)}(\bx) = 0
    \quad\mbox{for }\,\, j=1, \, \cdots, \, n_l.
 \end{equation}
 
To determine whether or not an element $K\in \cK^{(l-1)}$ is refined, for each $g_j^{(l)}(\bx)$, we compute its values at vertices of $K$. If these values change signs, then the element $K$ is partitioned by the hyper-plane $g_j^{(l)}(\bx)=0$ into two subdomains. It is possible that an element $K\in \cK^{(l-1)}$ may be partitioned by many hyper-planes in (\ref{new-hplanes}). Denote the collection of refined elements in $\cK^{(l-1)}$ by 
 \begin{equation}\label{refined-set}
     \cK^{(l-1)}_r =\left\{ K\in \cK^{(l-1)} \big|\, \exists\,   j_0 \mbox{ such that values of } g_{j_0}^{(l)}(\bx) \mbox{ at vertices of } K \mbox{ change signs} \right\}.
 \end{equation}
Denote by $\cK^{(l)}_K$ the physical partition of element $K\in \cK^{(l-1)}$ by hyper-planes in (\ref{new-hplanes}) and the boundary of $K$. Then the physical partition $\cK^{(l)}$ by the first $l$ hidden layers is given by
 \begin{equation}\label{K-l}
 \cK^{(l)} =\left(\bigcup\limits_{K\in \cK^{(l-1)}_r } \cK^{(l)}_K\right) \bigcup  \left(\cK^{(l-1)}\setminus \cK^{(l-1)}_r \right). 
 \end{equation}
The procedure of determining the physical partition of a function in ${\cal M}_N({\small\btheta},k)$ with $k\ge 3$ is summarized in Algorithm 3.1.
 
\begin{algorithm}\label{phy-partition}
{\bf {\sc \bf Algorithm 3.1}.} Physical Partition.\\
For any function $v\in {\cal M}_N({\small\btheta},k)$ with $k\ge 3$, the partition $\cK^{(1)}$ is determined by the boundary of the domain $\Omega$ and the hyper-planes $\left\{\bomega^{(1)}_i\cdot \bx-b^{(1)}_i=0\right\}_{i=1}^{n_1}$. For $l=2, \,\cdots,\,k-1$, 
\begin{itemize}
    \item[(1)] evaluate $g_j^{(l)}(\bx)$ at vertices of $\cK^{(l-1)}$ for $j=1, \, \cdots, \, n_l$;
    \item[(2)] determine $\cK^{(l-1)}_r$ using (\ref{refined-set});
    \item[(3)] for each $K\in \cK^{(l-1)}_r$, determine  its refinement by hyper-planes whose values change signs at vertices of $K$;
    \item[(4)] $\cK^{(l)}$ is given in (\ref{K-l}).
\end{itemize}
\end{algorithm}

\section{Adaptive Network Enhancement Method} 

Given a target function $f(\bx)$ and a prescribed tolerance $\epsilon>0$ for approximation accuracy, in \cite{LiuCai1} we proposed the adaptive network enhancement (ANE) method for generating a two-layer ReLU NN and a numerical integration mesh such that 
 \begin{equation}\label{tolerance}
     \|f(\cdot)-f_{_\cT}(\cdot ; {{\small\btheta}^*_{_\cT}})\| \leq \epsilon \, \|f\|,
 \end{equation}
where $f_{_\cT}(\bx ; {{\small\btheta}^*_{_\cT}})$ is the solution of the optimization problem in (\ref{L2App-d}) over a two-layer NN with numerical integration defined on the partition $\cT$. 

For the convenience of readers and needed notations, we state Algorithm 5.1 of \cite{LiuCai1} (see Algorithm 4.1 below) for the case that numerical integration based on a partition $\cT$ is sufficiently accurate. In this algorithm, $\cK$ is the physical partition of the current approximation $f_{_\cT}$, $\xi_{_K}= \|f-f_{_\cT}\|_{_{K,\cT}}$ is the local error in the physical subdomain $K\in\cK$, and the network enhancement strategy introduced in \cite{LiuCai1} consists of either the average marking strategy:
 \begin{equation}\label{BV-marking-2}
    \hat{\cK} =\left\{K\in \cK\, :\,
    \xi_{_K}
    \ge \, \dfrac{1}{\#\cK}\sum_{K\in \cK}\xi_{_K}\right\},
\end{equation}
where $\#\cK$ is the number of elements of $\cK$, or the bulk marking strategy: finding a minimal subset $\hat{\cK}$ of $\cK$ such that
\begin{equation}\label{marking-2}
    \sum_{K\in \hat{\cK}} \xi^2_{_K}
    \ge \gamma_1\, \sum_{K\in \cK} \xi^2_{_K}
    \quad\mbox{for }\,\, \gamma_1\in (0,\,1).
\end{equation}
With the subset $\hat{\cK}$, the number of new neurons to be added to the current NN is equal to $\#\hat{\cK}$, the number of elements in $\hat{\cK}$.

\begin{algorithm}\label{ane_old}
{\bf {\sc \bf Algorithm 4.1}.} Adaptive Network Enhancement for a two-layer NN with a fixed $\cT$.\\
Given a target function $f(\bx)$ and a tolerance $\epsilon >0$, starting with a two-layer ReLU NN with a small number of neurons,
 \vspace{2pt}
\begin{itemize}
    \item[(1)] solve the optimization problem in (\ref{L2App-d});
    \item[(2)] estimate the total error by computing $\xi= \left(\sum\limits_{K\in \cK} \xi^2_{_K}\right)^{1/2}\, / \|f\|_{_\cT}$;
     \vspace{2pt}
    \item[(3)] if $\xi< \epsilon$, then stop; otherwise, go to Step (4);
     \vspace{2pt}
    \item[(4)] add $\#\hat{\cK}$ neurons to the network, then go to Step (1). 
\end{itemize}
\end{algorithm}

For continuous functions exhibiting intersecting interface singularities or sharp transitional layer like discontinuities, numerical results in \cite{LiuCai1} showed the efficacy of the ANE method for generating a nearly minimal two-layer NN to learn the target function within the prescribed accuracy. However, in the case that the transitional layer is over a circle but not a straight line, the approximation by a two-layer NN with a large number of neurons exhibits a certain level of oscillation; while the approximation using a three-layer NN with a small number of parameters is more accurate than the former and has no oscillation. Those numerical experiments suggest that a three-layer NN is needed for learning certain type of functions even if they are continuous. 

In this section, we develop the ANE method for a multi-layer NN. To address question (b), i.e., when to add a new layer, we introduce a computable quantity denoted by $\eta_r$ measuring the improvement rate of two consecutive NNs per the relative increase of parameters. If the improvement rate $\eta_r$ is less than or equal to a prescribed expectation rate $\delta \in (0, 2)$, i.e.,
\begin{equation}\label{improve-rate}
    \eta_r \leq \delta,
\end{equation}
for two consecutive ANE runs, then the ANE method adds a new layer. Otherwise, the ANE adds neurons to the last hidden layer of the current network. Here a conservative strategy for adding a new layer is adopted by a double-run to check if inefficiency is identified when enhancing neurons in the current layer.

To define the improvement rate, denote the two consecutive NNs by $\cM_{ \scriptscriptstyle N^{\text{new}}}$ and $\cM_{ \scriptscriptstyle N^{\text{old}}}$, where the subscripts $N^{\text{new}}$ and $N^{\text{old}}$ are the number of parameters of these two NNs, respectively.
Assume that the former is obtained by adding neurons in the last hidden layer of the latter. Let $\xi^{\text{new}}$ and $\xi^{\text{old}}$ be the error estimators of the approximations using $\cM_{ \scriptscriptstyle N^{\text{new}}}$ and $\cM_{\scriptscriptstyle N^{\text{old}}}$, respectively. 
The improvement rate $\eta_r$ is defined as 
\[
\eta_r=\left(\frac{\xi^{\text{old}}-\xi^{\text{new}}}{\xi^{\text{old}}}\right)\Big/\left(\frac{(N^{\text{new}})^r-(N^{\text{old}})^r}{(N^{\text{new}})^r}\right),
\]
where $r$ is the order of the approximation with respect to the number of parameters and may depend on the activation function and the layer.

To determine the number of new neurons to be added in the last (but not first) hidden layer, our network enhancement strategy  starts with the marked subset $\hat{\cK}$ of $\cK$ as in the first layer, where $\cK$ is the physical partition of the current approximation. The subset $\hat{\cK}$ is further regrouped into a new set $\mathcal{C} = \{C:\, C \mbox{ is a connected, open subdomain of } \Omega\}$ such that each element of $\mathcal{C}$ is either an isolated subdomain in $\hat{\cK}$ or a union of connected subdomains in $\hat{\cK}$. Now, the number of new neurons to be added equals to the number of elements in $\mathcal{C}$. This strategy is based on the observation that 
a multi-layer NN is capable of generating piece-wise breaking hyper-planes in connected subdomains by one neuron. Summarizing the above discussion, our network enhancement strategy on adding neurons and layers is described in Algorithm 4.2.

\begin{algorithm}[htbp]\label{ane_new1}
{\bf {\sc \bf Algorithm 4.2}} Network Enhancement Strategy.\\
Given an error estimator $\xi$ and the improvement rate $\eta_r$,
 \vspace{2pt}
\begin{itemize}
    \item[(1)] if  (\ref{improve-rate}) holds for two consecutive ANE runs, add a new hidden layer; otherwise, go to Step (2);
    \vspace{2pt}
    \item[(2)] use the marking strategy in (\ref{BV-marking-2}) or (\ref{marking-2}) to generate $\hat{\cK}$, and regroup $\hat{\cK}$ to get $\cC$;
    \vspace{2pt}
    \item[(3)] if there is only one hidden layer, add $\#\hat{\cK}$ neurons to the first hidden layer; otherwise, add $\#\cC$ neurons to the last hidden layer.
\end{itemize}
\end{algorithm}

Assume that numerical integration on $\cT$ is accurate, then the ANE method for generating a nearly minimal multi-layer neural network is described in Algorithm 4.3.

\begin{algorithm}[htbp]\label{ane_new}
{\bf {\sc \bf Algorithm 4.3}} Adaptive Network Enhancement for a multi-layer NN with a fixed $\cT$.\\
Given a target function $f(\bx)$ and a tolerance  $\epsilon >0$ for accuracy, 
starting with a two-layer NN with a small number of neurons and using one loop of Algorithm 4.1 to generate a two-layer NN, then
 \vspace{2pt}
\begin{itemize}
    \item[(1)] solve the optimization problem in (\ref{L2App-d});
    \item[(2)] estimate the total error by computing $\xi= \left(\sum\limits_{K\in \cK} \xi_{_K}^2\right)^{1/2}\, / \|f\|_{_\cT}$;
     \vspace{2pt}
    \item[(3)] if $\xi< \epsilon$, then stop; otherwise, go to Step (4);
    \vspace{2pt}
    \item[(4)] compute the improvement rate $\eta_r$;
    \vspace{2pt}
    \item[(5)] add a new hidden layer or new neurons to the last hidden layer by Algorithm 4.2, then go to Step (1).
\end{itemize}
\end{algorithm}

\section{Initialization of training (iterative solvers)}
To determine the values of the network parameters, we need to solve the optimization problem in (\ref{L2App-d}), which is non-convex and, hence, computationally intensive and complicated. Currently, this problem is often solved by iterative optimization methods such as gradient descent (GD), Stochastic GD, Adam, etc. (see, e.g., \cite{BoCuNo2018} for a review paper in 2018 and references therein). Since non-convex optimizations usually have many solutions and/or many local minimum, it is then critical to start with a good initial guess in order to obtain the desired solution. 
As seen in \cite{LiuCai1}, the ANE method itself is a natural continuation process for generating good initializations. In this section, we discuss  initialization strategies of the ANE method for a multi-layer NN in two dimensions. Extensions to three dimensions are straightforward conceptually but more complicated algorithmically.

There are three cases that need to be initialized: (1) the beginning of the ANE method for a two-layer NN with a small number of neurons; (2) adding new neurons at the first layer; 
and (3) adding new neurons at the last hidden layer which is not the first layer. Initialization for both cases (1) and (2) was introduced in Section 5 of \cite{LiuCai1}. For the convenience of readers, we briefly describe them below.

The ANE method starts with a two-layer NN with $n_1$ neurons. Denote the input weights and bias by $\bomega^{(1)} = \left(\bomega^{(1)}_1, ..., \bomega^{(1)}_{n_1}\right)^T$ and $\bb^{(1)}=\left(b^{(1)}_1,..., b^{(1)}_{n_1}\right)^T$, respectively; and the output bias and weights by ${\bf c}^{(1)}=\left(b^{(2)},\omega_1^{(2)}, ..., \omega_{n_1}^{(2)}\right)^T$. The initials of $\bomega^{(1)}$ and $\bb^{(1)}$ are chosen such that the hyper-planes 
\[
\cP_i:\, \bomega_i^{(1)}\cdot \bx -b_i^{(1)}=0
\quad\mbox{for } i=1,...,n_1
\]
partition the domain uniformly. With the initial   $\btheta^{(1)} = (\bomega^{(1)}, {\bf b}^{(1)})$ prescribed above, let 
\[
\varphi_{0}^{(1)}(\bx)=1
\quad\mbox{and}\quad 
\varphi_{i}^{(1)}(\bx)=\sigma(\bomega_i^{(1)}\cdot \bx -b_i^{(1)})
\quad\mbox{for } i=1,...,n_1.
 \]
Then the initial of ${\bf c}^{(1)}$ is given by the solution of the following system of linear algebraic equations
\begin{equation}\label{O_W}
    \bm{M}(\btheta^{(1)})\, \bc^{(1)} = F(\btheta^{(1)}),
\end{equation}
where the coefficient matrix $\bm{M}(\btheta^{(1)})$ and the right-hand side vector $F(\btheta^{(1)})$ are given by
\[
\bm{M}\left(\btheta^{(1)}\right)=\left(\big(\varphi_{j}^{(1)}(\bx), \varphi_{i}^{(1)}(\bx)\big)\right)_{(n_1+1)\times (n_1+1)}
\quad\mbox{and}\quad
F\left(\btheta^{(1)}\right)=\left(\big(f, \varphi_{i}^{(1)}(\bx)\big)\right)_{(n_1+1)\times 1},
\]
respectively.

When adding new neurons at the first layer, the parameters associated with the old neurons will inherit the current approximation as their initials and those of the new neurons are initialized through the corresponding 
hyper-planes. Each new neuron is related to a 
sub-domain $K\in \hat{\cK}$ (see Algorithm 4.1) and is initialized by setting its corresponding hyper-plane to pass through the centroid of $K$ and orthogonal to the direction vector with the smallest variance of quadrature points in $K$. For details, see Section 5 of \cite{LiuCai1}.

In the case (3), new neurons are added either at a new layer or at the current but not the first layer. As in the case (2), the parameters of the old neurons are initialized with their current approximations. Below we describe our strategy on how to initialize newly added neurons. 
First, consider the case in which we add neurons to start a new layer. Assume that the current NN has $k-1$ hidden layers. By Algorithm 4.2, the number of new neurons to be added at the $k^{\text{th}}$ hidden layer equals to the number of elements in $\mathcal{C}^{(k-1)}$. For each element $C\in \mathcal{C}^{(k-1)}$, one neuron is added to the $k^{\text{th}}$ hidden layer, and its output weight is randomly initialized. Below we introduce a strategy to  initialize its bias and weights, $\bomega^{(k)}=\left(b^{(k)},\, \omega_1^{{(k)}}, \cdots, \omega^{{(k)}}_{n_{k-1}}\right)^T= \left(\omega_0^{{(k)}},\, \omega_1^{{(k)}}, \cdots, \omega^{{(k)}}_{n_{k-1}}\right)^T$. To this end, let us introduce a corresponding output function of the neuron to be added to refine $C$,
\[
l_C(\bx) = b^{(k)} + \sum^{n_{k-1}}_{i=1} \omega_i^{{(k)}} \sigma\left(\bomega_i^{(k-1)} \cdot \bx^{(k-2)} -b_i^{(k-1)}\right)
\equiv \sum\limits^{n_{k-1}}_{i=0}\omega_i^{\small{(k)}}\varphi_{i}^{(k-1)}(\bx),
\]
where the functions $\{ \varphi^{(k-1)}_i(\bx)\}_{i=0}^{n_{k-1}}$ are given by
\[
\varphi_{0}^{(k-1)}(\bx)=1
\quad\mbox{and}\quad 
\varphi_{i}^{(k-1)}(\bx)=\sigma\left(\bomega_i^{(k-1)} \cdot \bx^{(k-2)} -b_i^{(k-1)}\right).
 \] 
Note that $C\in \mathcal{C}^{(k-1)}$ is either an isolated physical subdomain or consists of several connected physical sub-domains in $\hat{\cK}^{(k-1)}$.

\begin{figure}[htbp]
\centerline{\includegraphics[width=3.5in]{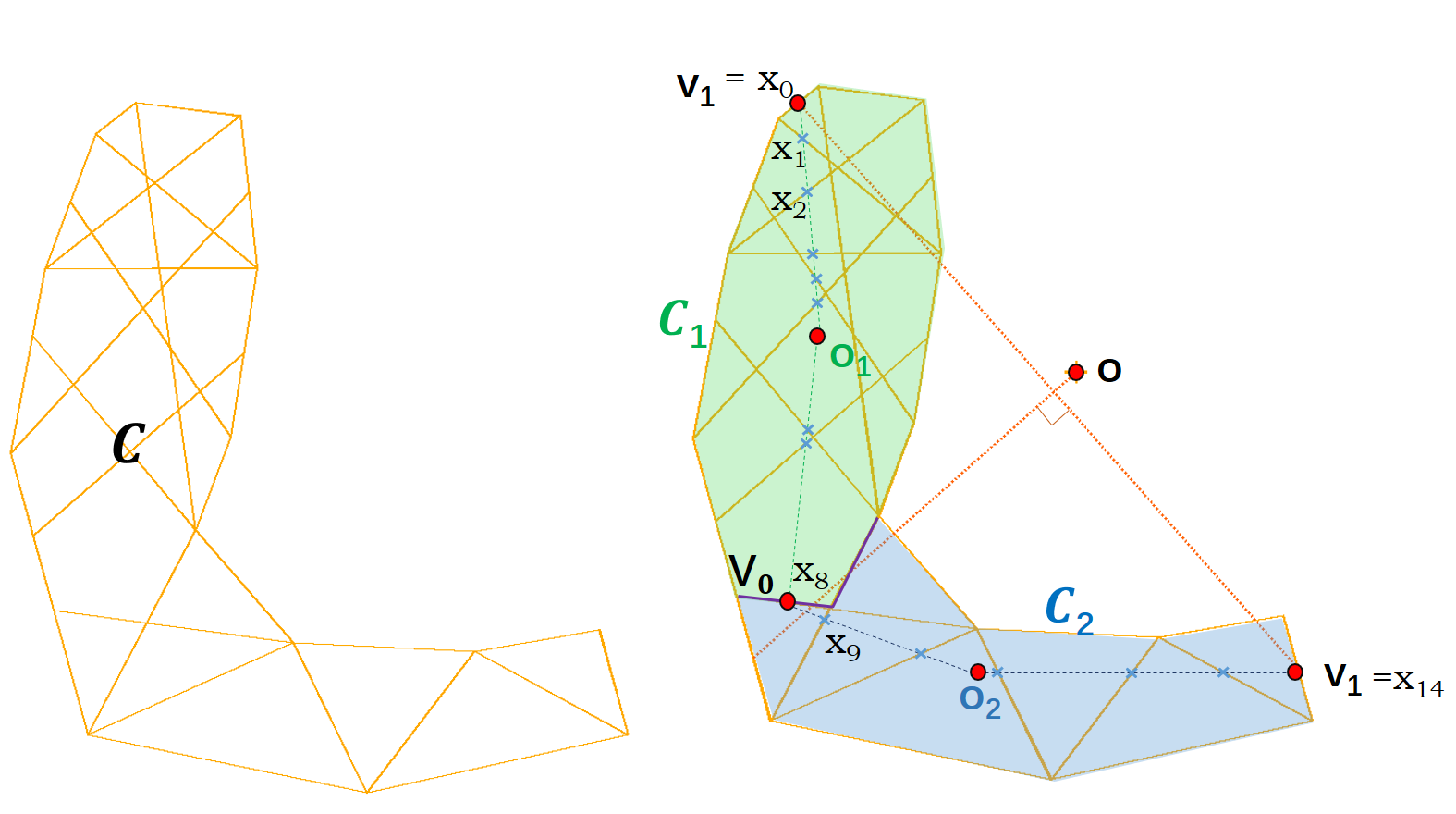}}
\caption{\small A heuristic method for initializing a neuron at the last hidden layer}
\label{f_nn_initial}
\end{figure}

A heuristic method is introduced here to initialize a neuron such that its corresponding break poly-lines can divide the sub-domains in $C$ as many as possible (please refer to Fig. \ref{f_nn_initial} for a graphical illustration): 
\begin{itemize}
    \item step 0: Initialize an empty set $X=\emptyset$; 
    \item step 1: Compute pairwise distances among mid-points on the boundary edges of $C$, find the point pair $(\bv_1, \bv_2)$ having the longest distance;
    \item step 2: Compute the centroid $\bo$ of $C$;
    \item step 3: If $\bo \in C$, find the set of intersection points of the edges in $C$ and the two line segments $\overline{\bo\bv_1}$ and $\overline{\bo\bv_2}$. Add the set of intersection points into $X$; 
    \item step 4: If $\bo \notin C$, decompose $C$ into two sub-regions $C_1$ and $C_2$ by the line passing through $\bo$ and perpendicular to $\overline{\bv_1\bv_2}$. Along the boundary edges of $C_1$ and $C_2$, locate a mid-point $\bv_0$ which has the largest distance sum to $\bv_1$ and $\bv_2$;  
    \item step 5: For each sub-region, use $(\bv_0, \bv_1)$ and $(\bv_0, \bv_2)$ respectively as the farthest point pair, repeat step 2-5 recursively until all sub-regions have their centroids located inside the region. 
\end{itemize}

The above procedure returns a point set $X = \{\bx_0, ..., \bx_{m}\}$.
Now, a reasonable initial is to choose $\{ \omega_i^{{(k)}}\}_{i=0}^{n_{k-1}}$ so that the corresponding function $l_C(\bx)$ vanishes at $\bx_j$ for all $0\leq j\leq m$, i.e., 
\begin{equation}\label{l_C}
    0=l_C(\bx_j) = \sum\limits^{n_{k-1}}_{i=0}\omega_i^{\small{(k)}}\varphi_{i}^{(k-1)}(\bx_j)
    =\bl_j^T \bomega^{(k)}
    \quad\mbox{for }\, j=0, 1, ..., m,
\end{equation}
where $\bx^{(0)}_j\!=\bx_j$, $\bx^{(k-2)}_j\! = N^{(k-2)} \circ \cdots \circ  N^{(1)}(\bx_j)$ for $k>2$, and $\bl_j\!=\!\left(\!1, \varphi_{1}^{(k-1)}(\bx_j),\cdots , \varphi_{n_{k-1}}^{(k-1)}(\bx_j)\!\right)^T$.
When $m< n_{k-2}$, any nontrivial solution of (\ref{l_C}) may serve as an initial of $\bomega^{(k)}$. When $m\ge n_{k-2}$, (\ref{l_C}) becomes an over-determined system and may not have a solution. In that case, we can choose a smaller $\gamma_1$ value in (\ref{marking-2}) so that fewer number of elements are marked. 

When neurons are added to the current layer in the case (3), 
the initialization procedure described above needs to be changed as follows. Note that each $C\in \mathcal{C}^{(k-1)}$
may be identified as a subset of $\tilde{C}$ consisting of $m$ connected physical sub-domains in $\mathcal{K}^{(k-2)}$. This implies that the $\{\bx_j\}_{j=0}^m$ in (\ref{l_C}) should be chosen based on the physical subdomains of $\tilde{C}$ in $\mathcal{K}^{(k-2)}$.


\section{Numerical Results for Learning Function}
This section presents numerical results of the ANE method for learning a given function through the least-squares loss function. The test problem is a function defined on the domain $\Omega = [-1,1]^2$ given by
\begin{equation}\label{test3}
f(x,y) = \tanh\left(\frac{1}{\alpha}(x^2 +y^2 - \frac{1}{4})\right) - \tanh\left(\frac{3}{4\alpha}\right),
 \end{equation}
which exhibits a sharp transitional layer across a circular interface for small $\alpha$. 
This test problem was used in \cite{LiuCai1} to motivate the ANE method for generating a multi-layer neural network. To learn $f(x,y)$ accurately, we show numerically that it is necessary to use at least a three-layer NN. The structure of a two- or three-layer NN is expressed as 2-$n_1$-1 or 2-$n_1$-$n_2$-1, respectively, where $n_i$ is the number of neurons at the $i^{\text{th}}$ hidden layer. 

In this experiment, we set $\alpha=0.01$ and the corresponding function $f$ is depicted in Fig. \ref{transition_adaptive} (a); a fixed $200\times 200 $ quadrature points are uniformly distributed in the domain $\Omega$; we use the bulk marking strategy defined in (\ref{marking-2}) with $\gamma_1 = 0.5$; and we choose the expectation rate $\delta =0.6$ with $r=1$ in (\ref{improve-rate}) and the tolerance $\epsilon=0.05$. The ANE method starts with a two-layer NN of 12 neurons, and the corresponding breaking lines $\{\cP_i\}_{i=1}^{12}$ are initialized uniformly. Specifically, half of breaking lines are parallel to the $x$-axis
\[\bomega_i^{(1)}=0 \quad \text{and}\quad b_i^{(1)}=-1+\frac{1}{3}i \quad\text{for}\quad i=0,\cdots, 5
\]
and the other half are parallel to the $y$-axis
\[\bomega_i^{(1)}=\pi/2 \quad \text{and}\quad b_i^{(1)}=-1+\frac{1}{3}(i-6) \quad\text{for}\quad  i=6,\cdots, 12.
\]
In addition, the output weights and bias are initialized by solving the linear system in (\ref{O_W}).


For each iteration of the ANE method, the corresponding minimization problem in (\ref{L2App-d}) is solved iteratively by the Adam version of gradient descent \cite{kingma2015} with a fixed learning rate $0.005$. The Adam's iterative solver is terminated when the relative change of the loss function $\|f-\hat{f}\|_{_\cT}$ is less than $10^{-3}$ during the last 2000 iterations. 

\begin{table}[htb!]
\caption{Adaptive numerical results for function with a transitional layer}
\label{transition_adaptive_result}
\centering
\begin{tabular}{  |l |c | c | c | c |  c| c|c| }
	\hline
	\makecell{Network structure}  & \# parameters & \makecell{Training accuracy\\ $\|f-\hat{f}\|_{_\cT}/ \|f\|$} 
	&\makecell{Improvement rate \\ $\eta$}\\[1mm] \hline
2-12-1 & 37 &0.357414 & --\\ \hline 
\textbf{2-18-1} & \textbf{55} &\textbf{0.323118} & \textbf{0.293198}\\ \hline
2-26-1 & 93 &0.272614 & 0.382528\\ \bottomrule
\textbf{2-18-5-1} &\textbf{137}  & \textbf{0.025483} &\textbf{1.538967}\\ \hline 
\end{tabular}
\end{table}

\begin{figure}[htbp]
\centering
\subfigure[The target function $f$ with a circular \newline transitional layer]{
\begin{minipage}[t]{0.45\linewidth}
\centering
\includegraphics[width=2.4in]{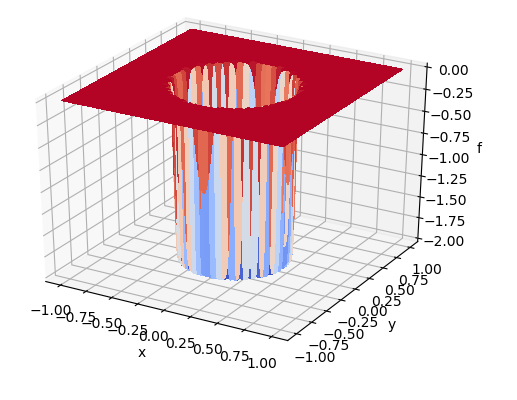}
\end{minipage}%
}%
\subfigure[PP of the approximation using 2-12-1 NN and centers of the marked elements (red dots)]{
\begin{minipage}[t]{0.45\linewidth}
\centering
\includegraphics[width=2.6in]{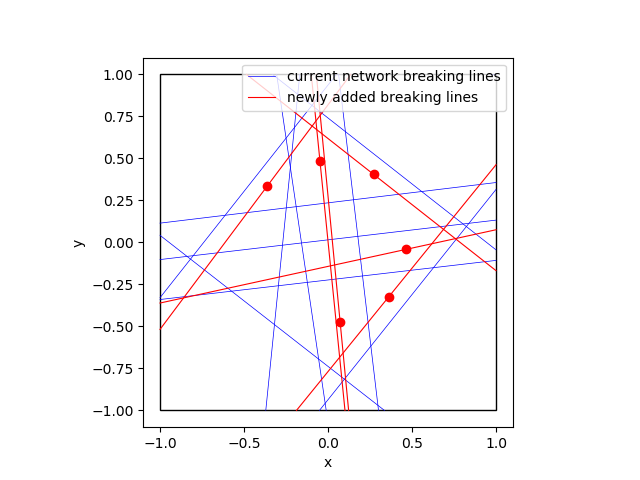}
\end{minipage}%
}%
\\
\centering
\subfigure[PP by 2-18-1 NN and isolated and connected sub-domains (dots)
]{
\begin{minipage}[t]{0.3\linewidth}
\centering
\includegraphics[width=1.7in]{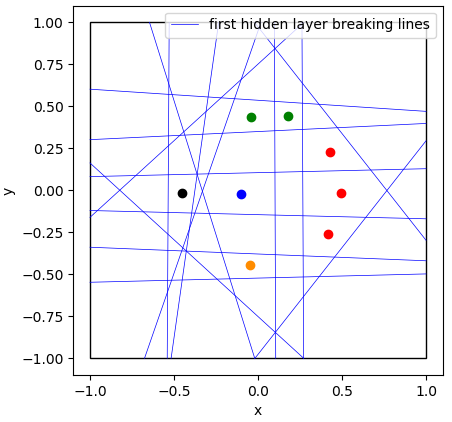}
\end{minipage}%
}%
\subfigure[PP by adaptive 2-18-5-1 NN]{
\begin{minipage}[t]{0.3\linewidth}
\centering
\includegraphics[width=1.72in]{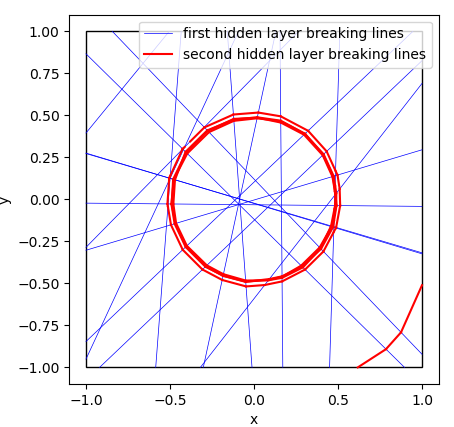}
\end{minipage}%
}%
\subfigure[Approximation using adaptive \newline 2-18-5-1 NN]{
\begin{minipage}[t]{0.35\linewidth}
\centering
\includegraphics[width=2.05in]{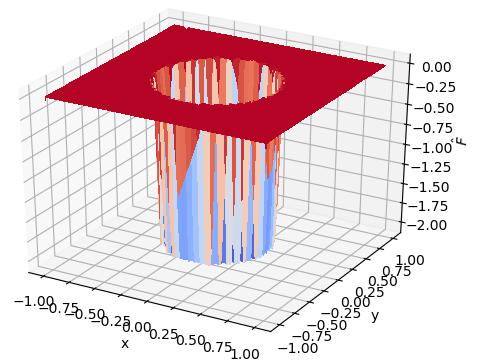}
\end{minipage}%
}%
\caption{Adaptive approximation results for function with a transitional layer}
\label{transition_adaptive}
\end{figure}

The ANE process is automatically terminated after four loops (see Table \ref{transition_adaptive_result}), and the final model of NN generated by the ANE is 2-18-5-1 with $137$ parameters.
The best least-squares approximation of the final NN model and the corresponding physical partition are depicted in Figs. \ref{transition_adaptive} (e) and (d). 
Clearly, the ANE method, using a relatively very small number of degrees of freedom, is capable of accurately approximating a function with thin layer without oscillation. This striking approximation property of the ANE method may be explained by the fact that the circular interface of the underlying function is captured very well by a couple of piece-wise breaking poly-lines of the approximation
generated by {\it the second layer}. 

Figs. \ref{transition_adaptive} (b)-(c) depict the physical partitions of the approximations at the intermediate NNs. In Fig \ref{transition_adaptive} (b), centers of the marked elements are illustrated by red dots; the breaking lines corresponding to the old and new neurons are displayed  by blue and red lines, respectively. Table \ref{transition_adaptive_result} shows that the adaptive network enhancement is done first at the current layer and then ended at the second hidden layer, because the improvement rates are smaller than the expectation rate for two consecutive network enhancement steps. Fig \ref{transition_adaptive} (c) shows that there are $8$ marked sub-domains and $5$ connected sub-domains, which explains only $5$ neurons are added at the second hidden layer.


\begin{table}[htb!]
\caption{Numerical results of adaptive and fixed NNs for function with a transitional layer}
\label{transition_fixed_result}
\centering
\begin{tabular}{  |l |c | c | c | c | c|c| }
	\hline
	\makecell{Network structure}  & \# parameters & \makecell{Training accuracy\\ $\|f-\hat{f}\|_{_\cT}/ \|f\|$} \\[1mm] \hline
2-18-5-1 (Adaptive) &137  & 0.025483\\ \hline 
2-18-5-1 (Fixed) & 137 &0.046199   \\ \hline
2-174-1 (Fixed) & 523 &0.111223 \\ \hline
\end{tabular}
\end{table}

\begin{figure}[htbp]
\centering
\subfigure[Approximation using fixed 2-174-1 NN]{
\begin{minipage}[t]{0.45\linewidth}
\centering
\includegraphics[width=2.6in]{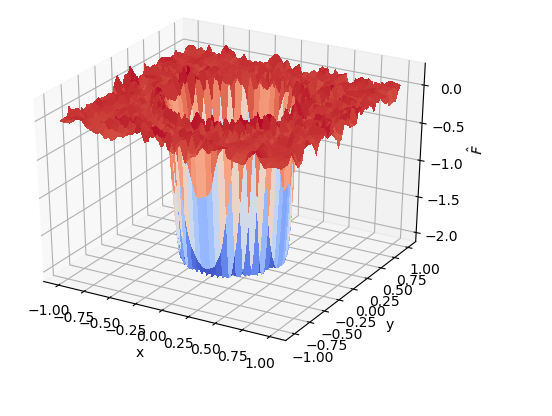}
\end{minipage}%
}%
\subfigure[PP of the approximation by 2-174-1 NN and centers of elements with large errors (red)]{
\begin{minipage}[t]{0.45\linewidth}
\centering
\includegraphics[width=2.6in]{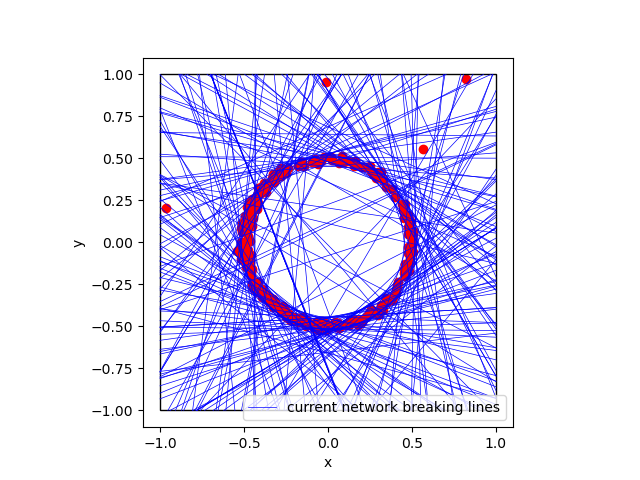}
\end{minipage}%
}%
\caption{Approximation results generated by a fixed {\em 2-174-1} NN for function with a transitional layer}
\label{trainsition_fixed_figure}
\end{figure}


For the purpose of comparison, in Table \ref{transition_fixed_result} we also report numerical results produced by two fixed NN models. With the same architecture of NN, the first two rows of Table \ref{transition_fixed_result} imply that the adaptive NN obtains a better training result than the fixed NN. This suggests that the ANE method does provide a good initialization. The second experiment uses a fixed one hidden layer with nearly four times more parameters than the adaptive NN; its approximation is less accurate (see the third row of Table \ref{transition_fixed_result}) and exhibits a certain level of oscillation (see Fig. \ref{trainsition_fixed_figure} (a)) which is not acceptable in many applications. Despite that the corresponding physical partition (see Fig. \ref{trainsition_fixed_figure}(b)) does capture the circular interface, it is too dense in the region where the function does not have much fluctuation. This experiment indicates that a three layer NN is necessary for approximating a function with thin layer. 

\section{Application to PDEs}

The ANE method introduced in this paper can be easily applied for learning solutions of partial differential equations. As an example, we demonstrate its application to the linear advection-reaction problem with discontinuous solution in this section. 

\subsection{Linear advection-reaction problem}

Let $\Omega$ be a bounded domain in ${\R}^d$ with
Lipschitz boundary, and $\bm{\beta}(\bx) = (\beta_1, \cdots, \beta_d)^T\in C^1(\bar{\Omega})^d$ be the advective velocity field. Denote the inflow part of the boundary $\partial \Omega$ by
\[
\Gamma_- = \{\bx\in\Gamma :\, \bm{\beta}(\bx) \cdot \bm{n}(\bx) <0\},
\]
where $\bm{n}(\bx)$ is the unit outward normal vector to $\Gamma_-$ at $\bx\in \Gamma_-$.
Consider the following linear advection-reaction equation 
\begin{equation}\label{pde}
    \left\{\begin{array}{rccl}
    u_{\bm\beta} + {\gamma}\, u &=&f &\text{ in }\, \Omega,\\[2mm]
    u&=&g &\text{ on }\,\, \Gamma_{-},
    \end{array}\right.
\end{equation}
where $u_{\bm\beta} = \bm{\beta}\cdot \nabla v$ is the directional derivative along the advective velocity field $\bm{\beta}$; and $\gamma \in C(\bar{\Omega})$, $f \in L^2(\Omega)$, and $g \in L^2(\Gamma_-)$ are given scalar-valued functions. 

Introduce the solution space of (\ref{pde}) and the associated norm as follows
\[V_{\bm\beta} = \{v\in L^2(\Omega): v_{\bm{\beta}}\in L^2(\Omega)\} \quad \text{and} \quad \vertiii{v}_{\bm\beta}= \left(\|v\|_{0,\Omega}^2 + \|v_{\bm\beta}
 \|_{0,\Omega}^2 \right)^{1/2},
\]
respectively. Define the least-squares functional by
 \begin{equation}\label{ls}
    \mathcal{L}(v;{\bf f}) = \|v_{\bm\beta} +{\gamma}\, v-f\|_{0,\Omega}^2 +  \|v-g\|_{-\bm\beta}^2 
\end{equation}
for all $v\in V_{\bm\beta}$, where ${\bf f} = (f,g)$ and the weighted norm over the inflow boundary is defined by
 \[
 \|v\|_{-\bm{\beta}} 
 =\left<v,v\right>^{1/2}_{-\bm{\beta}} 
 =\left( \int_{\Gamma_-} |\bm{\beta}\! \cdot \!\bm{n}|\, v^2\,ds\right)^{1/2}.
 \]
Now, the least-squares formulation of (\ref{pde}) (see, e.g., \cite{bochev2001improved, de2004least}) is to find $u\in V_{\bm\beta}$ such that 
\begin{equation}\label{LS}
    \mathcal{L}(u;{\bf f}) = \min\limits_{v\in V_{\bm\beta}} \mathcal{L}(v;{\bf f}).
\end{equation}

Assume that there exist a positive constant $\gamma_0$ such that
\begin{equation}\label{gamma}
    \gamma (\bx) - \frac{1}{2}\nabla \cdot \bm{\beta} (\bx) \geq \gamma_0 >0\quad  \text{ for all } \bx\in \Omega.
\end{equation}
It then follows from the trace, triangle, and Poincar\'{e} inequalities that the homogeneous LS functional $\mathcal{L}(v;{\bf 0})$ is equivalent to the norm $\vertiii{v}_{\bm\beta}^2$, i.e., there exist positive constants $\alpha$ and $M$ such that 
\begin{equation}\label{equiv}\alpha\, \vertiii{v}_{\bm\beta}^2 \leq \mathcal{L}(v;{\bf 0}) \leq M\, \vertiii{v}_{\bm\beta}^2.\end{equation}

\subsection{LSNN method and a posteriori error estimator}
 
Denote by $\cM_N({\small\btheta},l)$ the set of DNN functions as in Section 2. Let $\mathcal{T}$ be a partition of the domain $\Omega$ and ${\cal E}_{-}$ as a partition of the inflow boundary $\Gamma_-$. Let $\bx_{_K}$ and $\bx_{_E}$ be the centroids of $K\in {\cal T}$ and $E\in {\cal E}_-$, respectively. Then the least-squares neural network (LSNN) method introduced in 
\cite{LiuCai2} is to find $u_{_{\small {\cal T}}}^N(\bx,{\small\btheta}^*) \in \cM_N({\small\btheta},l)$ such that
\begin{equation}\label{lsnn_minimization}
\mathcal{L}_{_{\small {\cal T}}} \big({u}^N_{_{\small {\cal T}}}(\bx,{\small\btheta}^*);{\bf f}\big) 
  = \min\limits_{v\in \cM_N({\small\btheta},l)} \mathcal{L}_{_{\small {\cal T}}}\big(v(\bx;{\small\btheta});\,{\bf f}\big)
 = \min_{{\scriptsize \btheta}\in\R^{N}}\mathcal{L}_{_{\small {\cal T}}} \big(v(\bx; {\small\btheta});{\bf f}\big).
\end{equation}
where the discrete LS functional is given by
\[\mathcal{L}_{_{\small {\cal T}}}\big(v(\bx; {\small\btheta});{\bf f}\big) 
     = \sum_{K \in {\cal T}} \big(v_{\bm\beta} +{\gamma}\, v-f \big)^2(\bx_{_K}; {\small\btheta})\,|K|+  \sum_{E\in {\cal E}_-} \big(|\bm{\beta} \cdot \bm{n}|(v-g)^2\big)(\bx_{_E}; {\small\btheta})|E|.
\]
Here, $|K|$ and $|E|$ are the $d$ and $d-1$ dimensional measures of $K$ and $E$, respectively.

There are two key components in applying the ANE method: (a) an {\it a posteriori} error estimator for determining if the current approximation is within the prescribed tolerance and (b) {\it a posteriori} error indicators for determining how many new neurons to be added at either width or depth. 
As a gift from the LS principle, the value of the least-squares functional at the current approximation is a good {\it a posteriori} error estimator. Specifically, let $u_k\in \cM_N({\small\btheta},l)$ be the LSNN approximation at the current network and $u$ be the exact solution of (\ref{pde}), then the estimator is given by 
\begin{equation}\label{xi}
    \xi\equiv \sqrt{\cL_{_\cT}(u_k;\bff)}=\cL^{1/2}_{_\cT}(u_k;\bff).
\end{equation}
To estimate the relative error, we may use 
\[
\xi_{\text{rel}}=\dfrac{\cL^{1/2}_{_\cT}(u_k;\bff)}{\cL^{1/2}_{_\cT}(u_k;\bzero)}.
\]

\begin{lemma}
The estimator $\xi$ satisfies the following reliability bound:
\begin{equation}\label{reliability}
    \vertiii{u-u_k}_{\bm\beta}\leq 
 \dfrac{1}{\sqrt{\alpha}}\, \sqrt{\cL(u_k;\bff)} \leq \dfrac{1}{\sqrt{\alpha}}\,\xi + \mbox{h.o.t.},
\end{equation}
where h.o.t. means a higher order term.
\end{lemma}

\begin{proof}
The first inequality in (\ref{reliability}) is a direct consequence of the lower bound in (\ref{equiv}) and the fact that $\cL(u_k;\bff) = \cL(u-u_k;\bzero)$. The second inequality in (\ref{reliability}) follows from the fact that $\cL(u_k;\bff) = \cL_{_\cT}(u_k;\bff) + \mbox{h.o.t.}$ This completes the proof of the lemma.
\end{proof}

To define the local error indicators, we make use of the physical partition $\mathcal{K}^{(l-1)}=\{K\}$ of the current approximation (see Section 3). For each $K \in \mathcal{K}^{(l-1)}$, the indicator $\xi_K$ is defined by 
\begin{equation}\label{lsnn_estimator}
   \xi_K = \left(\|\left(u_k\right)_{\bm\beta} +{\gamma}\, u_k-f\|_{0,K}^2 +  \int_{\Gamma_-\cap \partial K} |\bm{\beta}\! \cdot \!\bm{n}|\, u_k^2\,ds\right)^{1/2}.
\end{equation}

\subsection{Numerical experiment}

In this section, we report numerical results for two test problems: (1) constant jump over two line segments and (2) non-constant jump over a straight line. In \cite{Cai2021linear}, we showed theoretically that a NN with at least two hidden layers is needed in order to accurately approximate their solutions. The purpose of this section is to demonstrate the efficacy of the ANE method for generating a nearly minimal NN to learn solutions of PDEs.

In both experiments, the integration is evaluated on a uniform partition of the domain with $100\times 100$ points. The prescribed expectation rate in (\ref{improve-rate}) is set at $\delta =0.6$ with $r=1$. For the iterative solver, a fixed learning rate $0.003$ and the same stopping criterion as that in Section 6 are used. Finally, the ANE method starts at a two-layer NN with initialization described in Section 6.

\subsubsection{Constant jump over two line segments}
The first test problem is the problem in (\ref{pde}) defined on $\Omega =(0,1)^2$ with $\gamma =f =0$ and a piece-wise constant advection velocity field
\[
\bm{\beta} =\left\{ \begin{array}{rclll}
&(1-\sqrt2,1)^T, & (x,y)\in \Upsilon_1=\{(x,y)\in\Omega:\, y<x\}, \\[2mm]
&(-1,\sqrt2-1)^T, & (x,y)\in \Upsilon_2=\{(x,y)\in\Omega:\, y\ge x\}.
 \end{array}\right.
\]
Denote the inflow boundary and its subset by
\[
\Gamma_-=\{(x,0):\, x\in (0,1)\} \cup \{(1,0)\}\cup \{(1,y):\, y\in (0,1)\}
\quad\mbox{and}\quad
\Gamma^1_-=\{(x,0): x\in (0,43/64)\},
\]
respectively. For the inflow boundary condition
\[
g(x,y)=\left\{ \begin{array}{rl}
 -1,& (x,y)\in \Gamma^1_-, \\[2mm]
 1, & (x,y)\in \Gamma^2_-=\Gamma_-\setminus \Gamma_-^1,
 \end{array}\right.
\]
the exact solution of the problem is $u=-1$ in $\Omega_1$ and $u=1$ in $\Omega_2=\Omega\setminus\bar{\Omega}_1$, where
 \[
 \Omega_1=\cup_{i=1}^2
 \{\bx\in\Upsilon_i: \bxi_i \cdot \bx < 43/64\},
 \,\, \bxi_1=(1, \sqrt{2}-1)^T,\mbox{ and }  \bxi_2=(\sqrt{2}-1, 1)^T.
\]
As depicted in Fig. \ref{transport_adaptive_fig}(a), the discontinuity of the solution is along two line segments. 

\begin{table}[htb]
\caption{Adaptive numerical results for the solution with a constant jump over two line segments}
\label{transport_adaptive_table}
\centering
\begin{tabular}{  |l |c |c | c | c |  c| c }
	\hline
	Network structure &
	$\#$ parameters & 
	${\dfrac{\|u-\bar{u}_{\tau}\|_0}{\|u\|_0}}$ &  
    \makecell{$\xi_{\text{rel}}=\frac{\mathcal{L}^{1/2}(\bar{u}_{\tau};\bf f)}{\mathcal{L}^{1/2}(\bar{u}_{\tau};\bf 0)}$} &
	\makecell{Improvement rate \\ $\eta$}\\[1mm] \hline
	2-6-1 &19  &0.543526 &0.462477  & --  \\ \hline 
	\textbf{2-7-1}&\textbf{22}  &\textbf{0.541213} &\textbf{0.449957}  &\textbf{0.328366}    \\ \hline 
	2-8-1&25  &0.545274 &0.449094  &0.022159    \\ \bottomrule
	2-7-1-1 &24  &0.515736 &0.401161  &2.120618   \\ \hline 
	2-7-2-1 &33  &0.510399 &0.391292  &0.159051   \\ \hline 
	2-7-3-1 &42  &0.113705 &0.066804  &4.510882   \\ \hline
	\textbf{2-7-4-1} &\textbf{51}  &\textbf{0.105822} &\textbf{0.019171}  &\textbf{5.197913} \\ \hline
\end{tabular}
\end{table}

\begin{figure}[htbp]
\centering
\subfigure[Exact solution $u$]{
\begin{minipage}[t]{0.32\linewidth}
\centering
\includegraphics[width=1.8in]{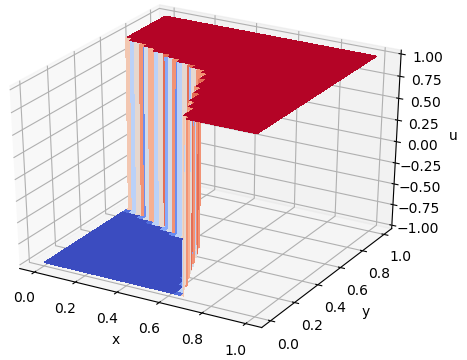}
\end{minipage}%
}%
\subfigure[PP by 2-6-1 NN, the marked \newline element (red dot), and new breaking \newline line (red line)]{
\begin{minipage}[t]{0.31\linewidth}
\centering
\includegraphics[width=1.6in]{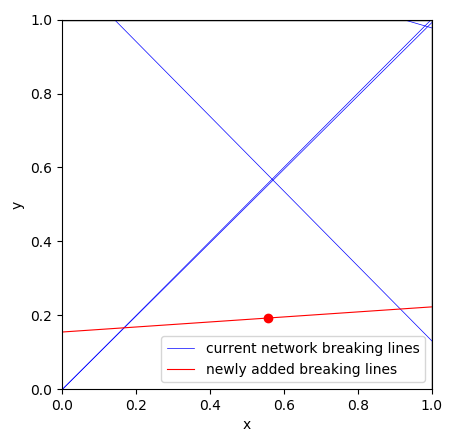}
\end{minipage}%
}%
\subfigure[Approximation by 2-7-1 NN]{
\begin{minipage}[t]{0.33\linewidth}
\centering
\includegraphics[width=1.8in]{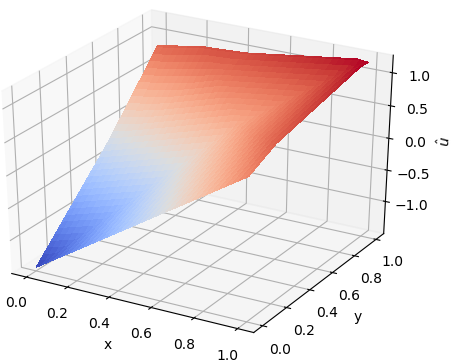}
\end{minipage}%
}%
\\
\centering
\subfigure[PP by 2-7-3-1 NN and the marked element]{
\begin{minipage}[t]{0.32\linewidth}
\centering
\includegraphics[width=1.7in]{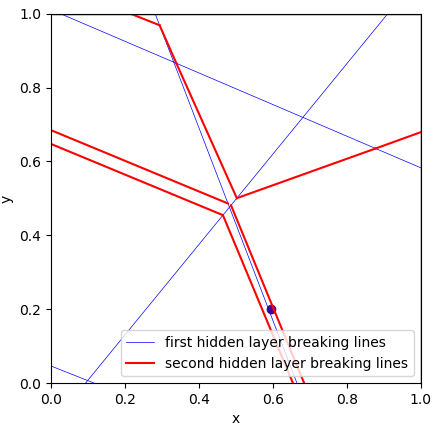}
\end{minipage}%
}%
\subfigure[PP by adaptive 2-7-4-1 NN]{
\begin{minipage}[t]{0.31\linewidth}
\centering
\includegraphics[width=1.67in]{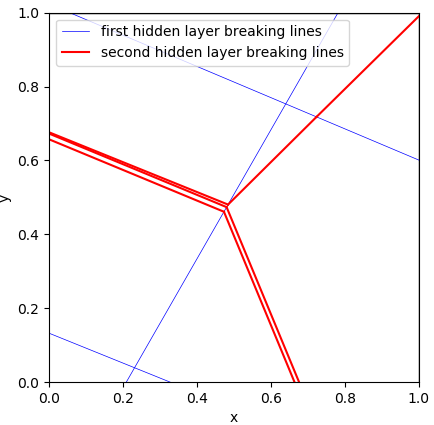}
\end{minipage}%
}%
\subfigure[Approximation using adaptive \newline 2-7-4-1 NN]{
\begin{minipage}[t]{0.33\linewidth}
\centering
\includegraphics[width=1.8in]{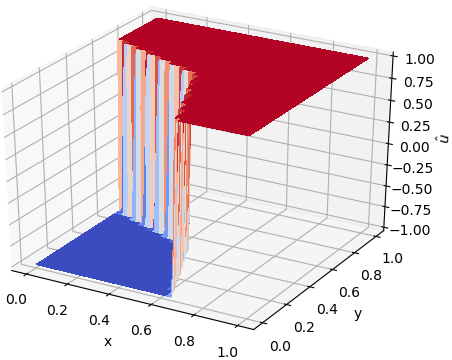}
\end{minipage}%
}%
\caption{Adaptive approximation results for the solution with a constant jump over two line segments}
\label{transport_adaptive_fig}
\end{figure}

Choosing $\gamma_1=0.6$ for the bulk marking strategy in (\ref{marking-2}), the ANE method is terminated when the relative error estimator $\xi_{\text{rel}}$ is less than the accuracy tolerance $\epsilon=0.05$. The architecture of the final NN model of this test problem is 2-7-4-1 with $\xi_{\text{rel}} = 0.019171 <\epsilon=0.05$ (see Table \ref{transport_adaptive_table}), and the corresponding approximation is depicted in Fig. \ref{transport_adaptive_fig} (f). Again, the corresponding  physical partition (see Fig. \ref{transport_adaptive_fig} (e)) accurately captures the interface by the piece-wise breaking lines of the second hidden layer. This explains why the ANE method produces an accurate approximation to a discontinuous solution without oscillation or overshooting.

Approximation results of intermediate NNs are also reported in Table \ref{transport_adaptive_table} and Fig. \ref{transport_adaptive_fig} (b)-(d). 
The second hidden layer is added when 
the improvement rate of two consecutive runs are less than the expectation rate (see the second and third rows in Table \ref{transport_adaptive_table}). 
Additionally, Fig. \ref{transport_adaptive_fig} (c) shows that a two-layer NN with seven neurons 
fails to approximate the discontinuous solution. This claim is actually true for a two-layer NN with 200 neurons (see \cite{Cai2021linear}). Hence, a three-layer NN is essential for learning the solution of this problem.

A fixed 2-7-4-1 NN is tested for a comparison. Due to random initialization of some parameters, the experiment is replicated 10 times. We observe from the training process that this fixed network gets trapped easily at a local minimum and fails to approximate the solution well in most of the duplicate runs. The best result is reported in Table \ref{transport_compare_table} and Fig. \ref{transport_compare_fig} (b). Although two network models have the same approximation power, attainable approximation may not be as accurate as the adaptive NN due to the inherent difficulty of non-convex optimization.

\begin{remark}
A fixed {\em 2-5-5-1} NN was employed for the same test problem in {\em \cite{Cai2021linear}}. Although a NN with fewer number of parameters can accurately approximate the solution, as pointed out in {\em Remark 5.1} in {\em \cite{Cai2021linear}} that the network gets trapped easily at a local minimum. Repeated training is necessary for a fixed network model. 
\end{remark}

\begin{table}[htb!]
\caption{Numerical results of adaptive and fixed NNs for the solution with a constant jump over two line segments}
\label{transport_compare_table}
\centering
\begin{tabular}{  |l |c |c | c | c |  c| c }
	\hline
	Network structure &
	$\#$ parameters & 
	${\dfrac{\|u-\bar{u}_{\tau}\|_0}{\|u\|_0}}$ &  
    \makecell{$\xi_{\text{rel}}=\frac{\mathcal{L}^{1/2}(\bar{u}_{\tau};\bf f)}{\mathcal{L}^{1/2}(\bar{u}_{\tau};\bf 0)}$} \\[1mm] \hline
	2-7-4-1 (Adaptive)  &51   &0.105822 &0.019171   \\ \hline 
	2-7-4-1 (Fixed)  &51   &0.164322 &0.116689   \\ \hline 
\end{tabular}
\end{table}

\begin{figure}[htbp]
\centering
\subfigure[Traces of the exact and numerical solutions on the plane $x=0$ by adaptive 2-7-4-1 NN]{
\begin{minipage}[t]{0.4\linewidth}
\centering
\includegraphics[width=2.1in]{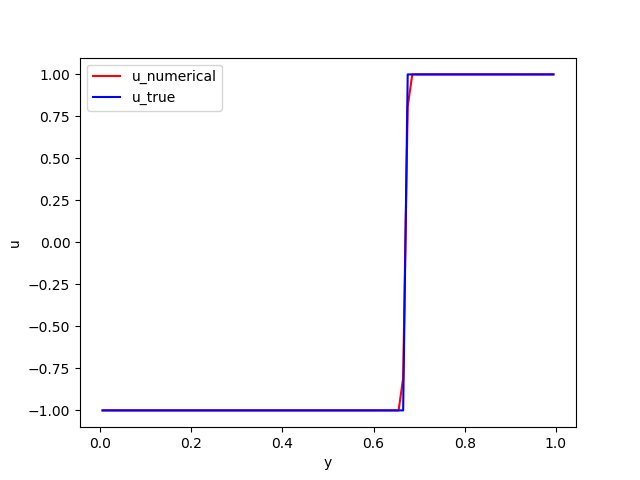}
\end{minipage}%
}%
\hspace{0.3in}
\subfigure[Traces of the exact and numerical solutions on the plane $x=0$ by fixed 2-7-4-1 NN]{
\begin{minipage}[t]{0.4\linewidth}
\centering
\includegraphics[width=2.1in]{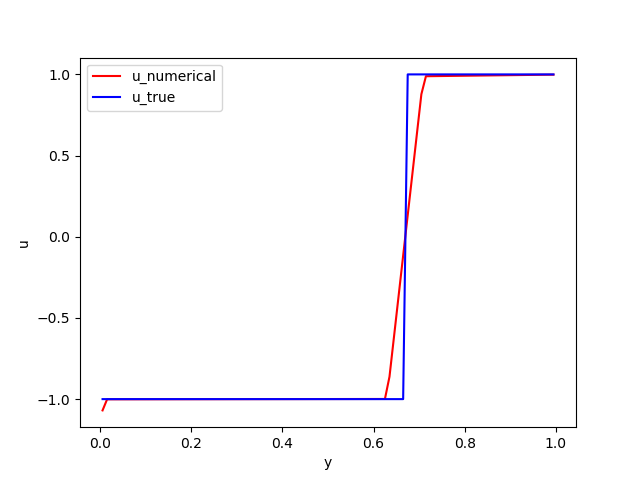}
\end{minipage}%
}%
\caption{Traces generated by adaptive and fixed NNs for the solution with a constant jump over two line segments}
\label{transport_compare_fig}
\end{figure}

\subsubsection{Non-constant jump over a straight line}The second test problem is again the equation in (\ref{pde}) defined on the domain $\Omega=(0,1)^2$ with a constant advection velocity field and a piece-wise smooth inflow boundary condition. Specifically, $\gamma =1$, $\bm{\beta} = (1,1)^T/\sqrt2$, and $\Gamma_-=\Gamma_-^1\cup \Gamma_-^2\equiv \{(0,y):\, y \in (0,1)\} \cup \{(x,0):\, x \in (0,1)\}$. Choose $g$ and $f$ accordingly such that the exact solution $u$ is
\[
u(x,y) = \left\{ \begin{array}{ll}
 \sin(x+y), & (x,y)\in \Omega_1=\{(x,y)\in (0,1)^2:\, y>x\}, \\[2mm]
 \cos(x+y), & (x,y)\in \Omega_2=\{(x,y)\in (0,1)^2:\, y<x\}.
 \end{array}\right.
 \]
As presented in Fig. \ref{smooth_adaptive_fig} (a), the interface of the discontinuous solution is the diagonal line $y=x$ and the jump over the interface is not a constant.
 
\begin{figure}[htbp]
\centering
\subfigure[Exact solution $u$]{
\begin{minipage}[t]{0.3\linewidth}
\centering
\includegraphics[width=1.7in]{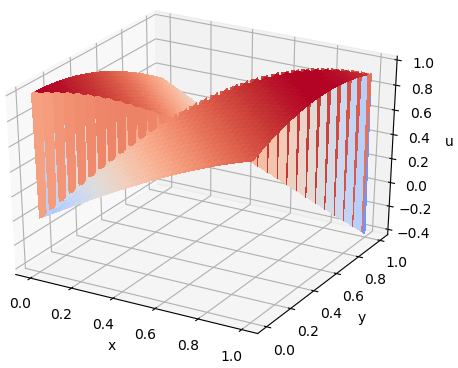}
\end{minipage}%
}%
\hspace{0.05in}
\subfigure[PP by 2-13-1 NN and elements with large errors (red dots)]{
\begin{minipage}[t]{0.3\linewidth}
\centering
\includegraphics[width=1.6in]{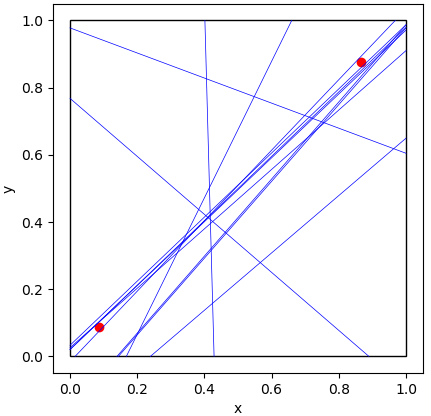}
\end{minipage}%
}%
\hspace{0.05in}
\subfigure[Traces of exact and numerical solutions on $y=1-x$ by 2-13-1 NN]{
\begin{minipage}[t]{0.3\linewidth}
\centering
\includegraphics[width=1.65in]{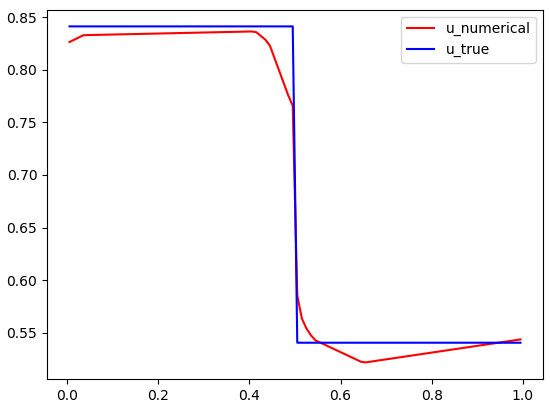}
\end{minipage}%
}%
\\
\centering
\subfigure[Approximation by adaptive 2-13-10-1 NN]{
\begin{minipage}[t]{0.3\linewidth}
\centering
\includegraphics[width=1.7in]{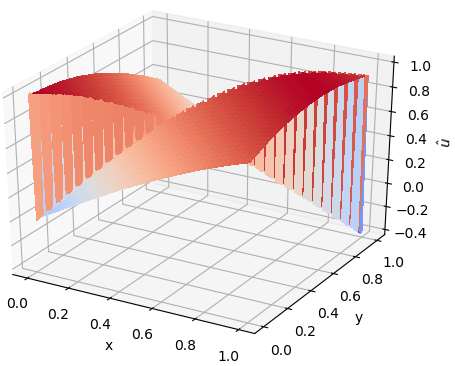}
\end{minipage}%
}%
\hspace{0.05in}
\subfigure[PP by adaptive 2-13-10-1 NN]{
\begin{minipage}[t]{0.3\linewidth}
\centering
\includegraphics[width=1.6in]{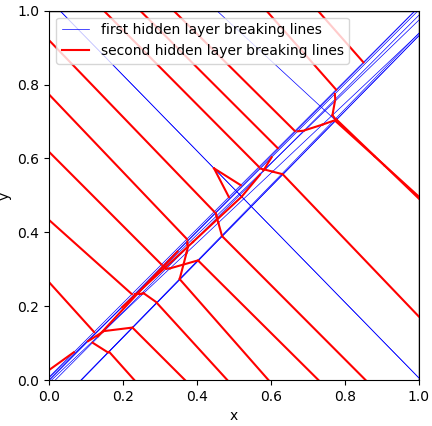}
\end{minipage}%
}%
\hspace{0.05in}
\subfigure[Traces of exact and numerical solutions on $y=1-x$ by adaptive 2-13-10-1 NN]{
\begin{minipage}[t]{0.3\linewidth}
\centering
\includegraphics[width=1.65in]{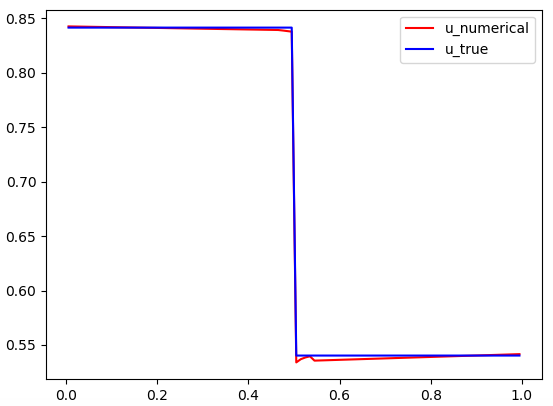}
\end{minipage}%
}%
\caption{Adaptive approximation results for the solution with a non-constant jump}
\label{smooth_adaptive_fig}
\end{figure}

Starting at a two-layer NN with six neurons and choosing $\gamma_1=0.3$ in the bulk marking strategy in (\ref{marking-2}), the ANE process repeats itself multiple runs until the accuracy tolerance $\epsilon=0.03$ is achieved. Ultimately, the ANE stops at a 2-13-10-1 NN model with the relative error estimator $\xi_{\text{rel}}=0.025733$ (see Table \ref{smooth_adaptive_table}). Fig. \ref{smooth_adaptive_fig} (d) and (e) illustrate the approximation and the corresponding physical partition using the final model. In addition, the traces of the exact and numerical solutions on the plane $y=1-x$ are depicted in Fig. \ref{smooth_adaptive_fig} (f), which clearly show that the final NN model is capable of accurately approximating the discontinuous solution without oscillation.

In Fig. \ref{smooth_adaptive_fig} (b)-(c), we also present the traces of the exact and numerical solution and the corresponding physical partition using an intermediate 2-13-1 NN. Again, this two-layer NN fails to provide a good approximation (see Fig. \ref{smooth_adaptive_fig} (c)) even though the corresponding physical partition (see Fig. \ref{smooth_adaptive_fig}(b)) locates the discontinuous interface. 
Moreover, as reported in Table \ref{smooth_compare_table}, the adaptive model yields to a better approximation result comparing to a fixed ReLU NN model of the same size.

\begin{table}[htb!]
\caption{Adaptive numerical results for the solution with a non-constant jump}
\label{smooth_adaptive_table}
\centering
\begin{tabular}{  |l |c |c | c | c |  c| c }
	\hline
	Network structure &
	$\#$ parameters & 
	${\dfrac{\|u-\bar{u}_{\tau}\|_0}{\|u\|_0}}$ &  
    \makecell{$\xi_{\text{rel}}=\frac{\mathcal{L}^{1/2}(\bar{u}_{\tau};\bf f)}{\mathcal{L}^{1/2}(\bar{u}_{\tau};\bf 0)}$} &
	\makecell{Improvement rate \\ $\eta$}\\[1mm] \hline
	2-6-1 &19  &0.085907 &0.178871  & --  \\ \hline
	2-8-1 &25  &0.075888 &0.157503  & 0.912494 \\ \hline
	2-10-1 &31  &0.070408 &0.135401  & 1.340723  \\ \hline
	\textbf{2-13-1}&\textbf{40}  &\textbf{0.070891} &\textbf{0.129806}  &\textbf{0.365856}    \\ \hline 
	2-15-1&46  &0.068234 &0.1250658  &0.522031   \\ \bottomrule
	2-13-2-1  & 57 &0.042813   & 0.100613 &1.290553\\ \hline 
2-13-4-1  & 87 &0.033823   & 0.091411 &0.505859\\ \hline 
2-13-7-1  & 132 &0.029862 & 0.065525 &1.429230 \\ \hline
2-13-9-1  & 162 &0.013429 & 0.044559 &2.883692  \\ \hline
\textbf{2-13-10-1} & \textbf{177} &\textbf{0.004651} & \textbf{0.025733} &\textbf{7.856341} \\ \hline
\end{tabular}
\end{table}

\begin{table}[htb!]
\caption{Numerical results of adaptive and fixed NNs for the solution with a non-constant jump}
\label{smooth_compare_table}
\centering
\begin{tabular}{  |l |c |c | c | c |  c| c }
	\hline
	Network structure &
	$\#$ parameters & 
	${\dfrac{\|u-\bar{u}_{\tau}\|_0}{\|u\|_0}}$ &  
    \makecell{$\xi_{\text{rel}}=\frac{\mathcal{L}^{1/2}(\bar{u}_{\tau};\bf f)}{\mathcal{L}^{1/2}(\bar{u}_{\tau};\bf 0)}$} \\[1mm] \hline
2-13-10-1 (Adaptive) & 177 &0.004651 & 0.025733 \\ \hline
2-13-10-1 (Fixed) &  177 &0.033602  &0.049884  \\\hline
\end{tabular}
\end{table}

\section{Conclusion}

Designing an optimal deep neural network for a given task is important and challenging in many machine learning applications. To address this important, open question, we have proposed the adaptive network enhancement method for generating a nearly optimal multi-layer neural network for a given task within some prescribed accuracy. This self-adaptive algorithm is based on the novel network enhancement strategies introduced in this paper that determine when a new layer and how many new neurons should be added when the current NN is not sufficient for the given task. This adaptive algorithm learns not only from given information (data, function, PDE) but also from the current computer simulation, and it is therefore a learning algorithm at a level which is more advanced than common machine learning algorithms.

The resulting non-convex optimization at each adaptive step is computationally intensive and complicated with possible many global/local minimums. The ANE method provides a natural process for obtaining a good initialization that assists training significantly.
Moreover, to provide a better initial guess, we have introduced an advanced procedure for initializing newly added neurons that are not at the first hidden layer.

In \cite{LiuCai1, LiuCai2} and this paper, we have demonstrated that the ANE method can automatically design a nearly minimal two- or multi-layer NN to learn functions exhibiting sharp transitional layers as well as continuous/discontinuous solutions of PDEs. Functions and PDEs with sharp transitions or discontinuities at unknown location have been a computational challenge when approximated using other functional classes such as polynomials or piecewise polynomials with fixed meshes.
In our future work, we plan to extend the applications of  self-adaptive DNN to a broader set of tasks such as data fitting, classification, etc., where training data is limited  but
given. The ANE method has a potential to resolve the so-called “over-fitting” issue
when data is noisy. 


 
\bigskip
\bibliographystyle{ieee}
\bibliography{main.bbl}

\end{document}